\documentclass[12pt,reqno]{amsart}
\usepackage{fullpage}
\usepackage{amsmath, mathtools,amsthm,dsfont}
\usepackage{amsfonts,lmodern,stix}
\usepackage{graphicx, float, caption, subcaption}
\usepackage{mathrsfs}

\usepackage[
    colorlinks=true,
    linkcolor=black,
    citecolor=black,
    urlcolor=blue
]{hyperref}

\usepackage[noabbrev,capitalize]{cleveref}
\usepackage[ddmmyyyy]{datetime}
\usepackage{enumerate}

\newcommand{\real}{\mathbb{R}}

\newcommand{\abs}[1]{\lvert#1\rvert}

\newcommand{\prob}{\mathbb{P}}
\newcommand{\eval}{\mathbb{E}}

\newcommand{\tildee}{\widetilde}

\newtheorem{thm}{Theorem}[section]

\newtheorem{lemma}[thm]{Lemma}
\newtheorem{prop}[thm]{Proposition}

\theoremstyle{definition}

\theoremstyle{remark}

\newtheorem{rmk}{Remark}

\allowdisplaybreaks

\numberwithin{equation}{section}

\begin{document}

\title{Empirical dual volumes}

\author{Grigoris Paouris}
\author{Peter Pivovarov}
\author{Paul Simanjuntak}

\maketitle
\begin{abstract}
  We introduce an empirical version of dual volumes within dual
  Brunn--Minkowski theory.  These provide a simple model for
  computation and analysis under symmetrization, including moment
  inequalities. As an application, we derive a new proof of the
  Busemann intersection inequality that relies on counting points in
  slabs rather than working directly with volumes of slices.
\end{abstract}


\section{Introduction}

\subsection{Intrinsic volumes and dual volumes}

Key functionals in integral geometry involve projections of convex
sets onto subspaces and an averaging process, typically using the
uniform measure on the Grassmanian. A central example is the
collection of intrinsic volumes.  These are defined by averaging the
volume of the orthogonal projection of a convex body $K$ in $\real^n$
onto $m$-dimensional subspaces of $\real^n$, namely
\begin{equation*}
  V_{m} (K) = {n \choose m} \frac{\omega_n}{\omega_m\omega_{n-m}}
  \int_{\mathcal{G}(n,m)} \abs{P_E K} \, dE,
\end{equation*}
where $\mathcal{G}(n,m)$ is the Grassmannian of $m$-dimensional subspaces of
$\real^n$ equipped with the normalized Haar measure, and $\omega_n$ is the
volume of the Euclidean unit ball in $\mathbb{R}^n$ (see \cite{Schneider_book}).

The notion of dual volumes was introduced by Lutwak in
\cite{Lu75}. These are defined for \emph{star bodies} $K$ in $\real^n$
and involve averaging of the volume of $m$-dimensional \emph{sections}
of $K$, i.e.,
\begin{equation*}
	\tildee{V}_{m} (K) = \frac{\omega_n}{\omega_{m}}
	\int_{\mathcal{G}(n,m)} \abs{K \cap E} \, dE.
\end{equation*}

As the intrinsic volumes involve projections of convex sets, they
naturally belong to the Brunn--Minkowski theory
\cite{Schneider_book}. Within the class of convex sets, projections
and sections are dual notions. Since dual volumes involve sections of
star bodies, they often reflect duality principles that are not
explained by convexity. They have come to form a foundational part of
dual Brunn--Minkowski theory, or simply {\it the dual theory} e.g.,
\cite{Lu75, Lutwak88,LZ97,G07, HLYZ, LYZ18}. The functionals $V_m$ and
$\tildee{V}_m$ are also known as quermassintegrals and dual
quermassintegrals, respectively, under an alternate indexing:
\begin{equation}
  \label{eq:alt_index}
V_{m} (K) = \binom{n}{m}\frac{1}{\omega_{n-m}}W_{n-m} (K), \qquad
\tildee{V}_{m} (K) = \tildee{W}_{n-m} (K).
\end{equation}

Intrinsic volumes and dual volumes both admit isoperimetric
inequalities in which Euclidean balls are extremal, subject to a
volume constraint.  The fundamental Alexandrov-Fenchel inequality
(e.g., \cite{Schneider_book}) implies that for convex bodies $K$ in
$\mathbb{R}^n$,
\begin{equation}
	\label{eq:afi}
	V_{m} (K) \geq V_{m} (K^*),
\end{equation}
where $K^*$ is an Euclidean ball with the same volume as $K$.  The
dual inequality inequality was first proved by Lutwak in \cite{Lu75},
where the direction is reversed, and holds for star bodies $K$ in
$\mathbb{R}^n$,
\begin{equation}
	\label{eq:dmv}
	\tildee{V}_{m} (K) \leq \tildee{V}_{m} (K^*).
\end{equation}
The similarity (and reversal) of these inequalities represents a
recurring theme between notions in the Brunn--Minkowski theory and the
dual theory, even though techniques for their analysis can be quite
different.

Inequality \eqref{eq:afi} and its special cases, including the
standard isoperimetric theorem (surface area corresponds to $m=n-1$),
have been studied from different probabilistic perspectives. A common
stochastic model, with a long history in geometric probability,
involves the convex hull of $N$ independent random points $X_i$
sampled uniformly in a convex body $K$, denoted here by
\begin{equation*}
[K]_N=\mathop{\rm
  conv}\{X_1,\ldots,X_N\}.
\end{equation*}
In this case, the expectations satisfy
\begin{equation}
  \label{eqn:KN}
  \mathbb{E} V_{m}([K]_N)\geq \mathbb{E} V_{m}([K^*]_N),
\end{equation}which recovers \eqref{eq:afi} when $N\rightarrow \infty$.
The roots of \eqref{eqn:KN} are in Blaschke's resolution of
Sylvester's problem \cite{SW}, with extensions for the volume of
simplices by Busemann \cite{Bus:1953}, polytopes by Groemer
\cite{Groe:1974}, while intrinsic volumes were treated by Pfiefer in
\cite{Pfie:1982} (see also \cite{ CCG:1999, HarPao:2003, PP12} and
\cite[Chapter 10]{Schneider_book} and the references therein). The
convex hull operation is just one natural mode of approximating convex
sets. A number of randomized analogues of fundamental affine
constructions of convex sets, including polar bodies, centroid bodies
and central objects in Lutwak--Yang--Zhang's $L_p$-Brunn--Minkowksi
theory (e.g, \cite{LYZ00,LZ97}), admit stronger stochastic
inequalities, surveyed in \cite{PP17}.



Since the dual theory involves star-shaped sets, random {\it convex}
approximations are of limited use.  Fundamental descriptions of
star-bodies as limits of radial sums of ellipsoids from
\cite{GooWei95, GZ99, KKYY07}, were recently adapted for stochastic
approximation in \cite{APPS} to treat non-convex isoperimetric
problems for dual $L_p$-centroid bodies. Beyond approximations of
star-shaped sets themselves, it is natural to seek approximations of
{\it functionals} of star bodies. Our focus here is to develop a
stochastic model for dual volumes, establish extremal inequalities,
and revisit duality features probabilistically.

\subsection{Empirical dual volumes}

To ground integral geometric constructions involving subspaces of
$\real^n$, we recall a connection between the uniform distribution on
the Grassmanian and Gaussian matrices.  Let $G$ be an $m \times n$
Gaussian matrix with independent standard normal entries. Tsirelson's
Gaussian representation of intrinsic volumes
\cite{Tsirelson1986GeomMLE2} can be expressed as
\begin{equation}
  \label{eq:Tsirel_original}
  V_{m} (K) = \frac{\omega_n}{\omega_{n-m}}\frac{\eval_G \abs{G K}}{\eval_G
    \mathop{\rm det}(GG^T)^{1/2}},
\end{equation}
where $G^{T}$ denotes the transpose of $G$. By analogy, we introduce a
definition of the dual volume that is complementary to Tsirelson's model. There
are two natural choices, depending on the dimension $m$ or codimension
$\ell=n-m$, as in the indexing of dual quermassintegrals in \eqref{eq:alt_index}.

\begin{prop}
  \label{prop:Tsirelson}Let $m,n\in \mathbb{N}$ with $m<n$.
  Let $G$ be an $m \times n$ Gaussian matrix with independent standard
  normal entries and define
  \begin{equation}
    \label{eqn:Delta}
    \Delta_{m}=\eval_G \mathop{\rm det}(GG^T)^{-1/2}.
  \end{equation}
  Let $K$ be a star body in $\real^n$. Then
  \begin{equation}
    \label{eq:dqmi-def1}
\tildee{V}_{m} (K) = \frac{\omega_n}{\omega_{m}}\frac{\eval_G |G^{-T}
  K|}{\Delta_m}=
\frac{\omega_n}{\omega_{m}}\eval_G \abs{K \cap
\mathop{\rm{Im}} \, G^T}
\end{equation}and for $\ell=n-m$,
\begin{equation}
\label{eq:dqmi-def2}
\tildee{V}_{\ell} (K) = \frac{\omega_n}{\omega_{\ell}}\eval_G |K \cap
\mathop{\rm{ker}} G|.
\end{equation}
\end{prop}

Working with Gaussian matrices, specifically \eqref{eq:dqmi-def2},
provides a convenient way to compute $\tildee{V}_{\ell} (K)$ in terms
of the rows $g_1,\ldots,g_m$ of $G$, which are independent. Namely, we
approximate the expected volume $\eval_G \abs{ K \cap \ker(G)}$
as follows:


\begin{itemize}
\item[{\bf (i)}] replace $K$ by independent random vectors $X_1,\ldots,X_N$
  uniformly distributed in $K$;
\item[{\bf (ii)}] replace the kernel
  \begin{equation*}
  \mathop{\rm ker}(G)=\bigcap_{j=1}^m\{x\in
  \mathbb{R}^n:\langle x,g_j\rangle =0\}
  \end{equation*}
  by 
  the intersection of Gaussian slabs with parameter $t>0$,
  \begin{equation}
    \label{eq:HGt}
    H_{G,t}=\bigcap_{j=1}^m\left\{x\in \mathbb{R}^n: \lvert\langle
    x,g_j\rangle \rvert \leq t \right\};
  \end{equation}
  \item[{\bf (iii)}] use the number of points $X_i$ landing in $H_{G,t}$ as a
    proxy for the volume of $K\cap H_{G,t}$.
\end{itemize}
With this in mind, we define for a fixed $N \in \mathbb{N}$ and $t >
0$, the empirical $\ell$-th dual volume of a star body $K$ with
$\ell=n-m$ by
\begin{equation}
\label{eq:edqmi-def}
  \tildee{V}_{\ell,t,N} (K) = \frac{\omega_n }{\omega_{\ell}}
  \mathbb{E}_G\left(\frac{\lvert K \rvert}{N}\sum_{i=1}^N \mathds{1}_{H_{G,t}}(X_i)\right).
\end{equation}
As we will show (Proposition \ref{lem:vm-gaussian}), scaling by
$\Delta_m$, the random variable $\tildee{V}_{\ell,t,N}(K)$ converges
almost surely to the $\ell$-th dual volume:
\begin{equation}
  \label{eqn:lim}
  \lim_{t \to 0} \lim_{N \to \infty}
  \frac{ \tildee{V}_{\ell,t,N} (K) }{(2t)^m \Delta_m}= \tildee{V}_{\ell} (K).
\end{equation}

Our first theorem establishes monotonicity of $\tildee{V}_{\ell,t,N}
(K)$ under Steiner symmetrization of $K$ in an arbitrary direction
$u$, denoted here by $S_uK$ (see \S \ref{sec:prelim} for definitions).
As $\tildee{V}_{\ell,t,N}(K)$ is a random variable, it will be useful
to provide comparisons via stochastic ordering (see \cite{SS12}): for
non-negative random variables $\xi,\eta$ we write $\xi \prec \eta$ if
\begin{equation}
  \label{eq:stoch_dom}
  \mathbb{P}(\xi > s)\leq \mathbb{P}(\eta > s)\quad \text{ for all }s\geq 0.
\end{equation}With this notation, we have the following result.

\begin{thm}
  \label{thm:eqmi}
  Let $K$ be a star body in $\real^n$, $\ell< n$, and let $u\in
  S^{n-1}$. Then for each $N\in \mathbb{N}$ and each $t>0$,
  \begin{equation}
    \label{eqn:dom_u}
    \tildee{V}_{\ell,t,N} (K) \prec \tildee{V}_{\ell,t,N} (S_uK),
  \end{equation}and, consequently,
  \begin{equation}
    \label{eqn:dom_*}
    \tildee{V}_{\ell,t,N} (K) \prec \tildee{V}_{\ell,t,N} (K^*).
  \end{equation}
\end{thm}

By the limiting relation \eqref{eqn:lim}, the stochastic dominance in
\eqref{eqn:dom_u} underlies the inequality
\begin{equation*}
  \tildee{V}_{\ell}(K)\leq \tildee{V}_{\ell} (S_uK).
\end{equation*}
Applying successive Steiner symmetrizations, this also leads to a
stochastic approach to \eqref{eq:dmv}, based on enumeration of random
points in the slabs $H_{G,t}$.

\subsection{Higher moment inequalities}

We also obtain a result about higher order integer moments. As above,
we let $X_1,\ldots,X_N$ be independent random vectors uniformly
distributed in a star body $K\subseteq\mathbb{R}^n$. For $N\in
\mathbb{N}$, $t>0$, and $q\in \mathbb{N}$, we define
\begin{equation}
\label{eq:edqmi-def-q}
  \tildee{V}_{\ell,t,N} (K,q) = \frac{\omega_n }{\omega_{\ell}}
  \left(\mathbb{E}_G\left(\frac{\lvert K \rvert}{N}\sum_{i=1}^N
  \mathds{1}_{H_{G,t}}(X_i)\right)^q\right)^{1/q}.
\end{equation}
We emphasize that the $q$-th power is taken  before expectation in
$G$.

\begin{thm}
  \label{thm:moments}
  Let $K$ be a star body in $\real^n$, $\ell<n$, and let $u\in
  S^{n-1}$. Then for $N\in \mathbb{N}$, $t>0$, and $q\in \mathbb{N}$,
  \begin{equation}
    \eval \tildee{V}^q_{\ell,t,N} (K,q) \leq
        \eval \tildee{V}^q_{\ell,t,N} (S_uK,q),
  \end{equation}and, consequently,
  \begin{equation}
    \eval \tildee{V}^q_{\ell,t,N} (K,q) \leq
        \eval \tildee{V}^q_{\ell,t,N} (K^*,q).
  \end{equation}
    
\end{thm}

This theorem has a direct connection to the Busemann intersection
inequality \cite{Bus:1953}, a fundamental precursor to central notions
in the dual theory. If $K$ is a star body in $\mathbb{R}^n$, we set
\begin{equation*}
\tildee{\Phi} (K) = \left( \int_{S^{n-1}} |K \cap \theta^\perp|^n \,
d\theta \right)^\frac{1}{n}.
\end{equation*}The Busemann intersection inequality asserts that
\begin{equation}
  \label{eqn:Busemann}
  \tildee{\Phi} (K)\leq \tildee{\Phi} (K^*).
\end{equation} The original proof of \eqref{eqn:Busemann} involved Steiner
symmetrization -- but not in the typical sense of showing $\Phi(K)\leq
\Phi(S_uK)$ for any direction $u$ (see \cite{Bus:1953, G07, Gar06}).
Proofs of \eqref{eqn:Busemann} that make direct use of Steiner
symmetrization applied to $K$ have only been found in the last several
years, in \cite{APPS} and a new variational approach to symmetrization
inequalities by Milman, Shabelmann, and Yehudayoff \cite{MSY}. This is
one indication of the somewhat sporadic use of Steiner symmetrization
in the dual theory, in contrast to its wide use in the
Brunn--Minkowski theory (e.g. \cite[Chapter 10]{Schneider_book}).

Specializing to $m=1$ in Theorem \ref{thm:moments}, we note that $G$ is
simply a $1\times n$ row vector, say $g$. In this case we write $\eval_g$
instead of $\eval_G$ and set $H_{g,t}=\{x\in \mathbb{R}^n:\lvert
\langle x, g\rangle \rvert \leq t\}$. We single out the case $q=n$,
and define for $N\in \mathbb{N}$ and $t>0$,
\begin{equation}
\label{eq:emp-adqmi}
\tildee{\Phi}_{t,N} (K) =  \left(\eval_g \left( \frac{|K|}{N} \sum_{i=1}^N 
	  \mathds{1}_{H_{g,t}}(X_i) \right)^n\right)^{1/n}.
\end{equation}

The connection to Busemann's theorem \eqref{eqn:Busemann} is through
the following result, which includes a subtle limiting property of
$\tildee{\Phi}_{t,N}(K)$.

\begin{thm}
  \label{thm:ebie}
  Let $K$ be a star body in $\real^n$ and let $u\in S^{n-1}$. Then
  \begin{equation}
    \label{eqn:bie-steiner} 
    \eval  \tildee{\Phi}^n_{t,N} (K) \leq \eval \tildee{\Phi}^n_{t,N} (S_uK)
  \end{equation}and, consequently,
  \begin{equation}
    \label{eq:bie-emp}
    \eval \tildee{\Phi}^n_{t,N} (K)
    \leq \eval  \tildee{\Phi}^n_{t,N} (K^*).
  \end{equation}
  Moreover, for $c_n=n\omega_n/(2\pi)^{n/2}$, 
  \begin{equation}
    \label{eqn:ebie_lim}
    \lim_{t\rightarrow 0^{+}}\lim_{N\rightarrow \infty} \frac{\eval
      \tildee{\Phi}_{t,N} (K)}{c_n(2t)^n\ln (1/t)} = \tildee{\Phi}(K).
  \end{equation}
\end{thm}

Thus Theorem \ref{thm:ebie} directly implies \eqref{eqn:Busemann} and
the proof relies on enumerating points in the Gaussian slabs
$H_{g,t}$. The methods we use here are similar to those of
\cite{APPS}, but the functional $\tildee{\Phi}_{t,N}(\cdot)$ is
conceptually simpler.  We note that when $K \subseteq \real^2$,
\eqref{eqn:Busemann} is an equality as $\tildee{\Phi}(K)$ is just a
multiple of the area of $K$. However, the empirical version
\eqref{eq:bie-emp} is a non-trival new inequality in the plane.

\section{Preliminaries}
\label{sec:prelim}

We call a set $K$ in $\real^n$ star-shaped if the origin is an
interior point and the intersection of $K$ with any ray from the
origin is a closed line segment. If $K$ is star-shaped, its radial
function $\rho_K: S^{n-1} \to \real$ is defined by
\begin{equation*}
\rho_K (\theta) = \sup\{ \lambda > 0 \, | \,
	\lambda \theta \in K \}.
\end{equation*}
We call $K$ a star body if it is a compact star-shaped set having a
continuous radial function. The radial function can be extended
$(-1)$-homogeneously to $\real^n\backslash\{0\}$ and we will denote
its extension in the same way. When $K$ is a star body, its volume can
be expressed using polar coordinates (e.g.,
\cite[pg. 57]{Schneider_book}) as
\begin{equation}
  \label{eqn:polar}
\abs{K} = \omega_n \int_{S^{n-1}} \rho_K^n (\theta) \, d\theta.
\end{equation}
Here, as above, $\omega_n$ denotes the volume of the $n$-dimensional
unit Euclidean ball; we use $\abs{\cdot}$ to denote the Lebesgue
measure, with the dimension determined by the ambient space.  If $K$
is an origin-symmetric convex set, the polar of $K$ is defined by
$$K^\circ = \{x \in \real^n : \abs{\langle x, y
\rangle} \leq 1 \text{ for all } y \in K\}.$$

Intrinsic volumes can alternatively be defined as coefficients in the
Steiner formula (\cite{Schneider_book},\cite{Gar06}), for convex
bodies $K$ in $\mathbb{R}^n$ and $\varepsilon>0$,
\begin{equation*}
  \abs{K+\varepsilon B_2^n}
  =\sum_{m=0}^{n}
  \omega_{n-m}\,V_{m}(K)\,\varepsilon^{n-m}.
\end{equation*}

A central notion in dual Brunn--Minkowski theory is radial summation
\cite{Lu75}. For star bodies $K$ and $L$ their radial sum $K
\tildee{+} L$ is defined by
\begin{equation*}
\rho_{K \tildee{+} L} (\theta)
= \rho_K (\theta) + \rho_L (\theta)
\qquad
\theta \in S^{n-1}.
\end{equation*}
Dual volumes arise as coefficients in the {\it dual Steiner formula}
(see \cite{Lu75,Gar06}), for star bodies $K$ in $\mathbb{R}^n$ and
$\varepsilon >0$,
\begin{equation*}
\abs{K \,\tildee{+}\, \varepsilon B_2^n} = \sum_{m=0}^{n}
 {n \choose m}\tildee{V}_m(K)\, \varepsilon^{\,n-m}.
\end{equation*}
We will also use the following equivalent expressions for dual
volumes (via polar coordinates),
\begin{align}
  \tildee{V}_{m} (K) &= \omega_n \int_{\mathcal{G}(n,m)} \int_{S^{n-1} \cap E}
  \rho_K^{m} (\theta) \, d\theta \, dE \nonumber \\ & =
  \omega_n\int_{S^{n-1}}\rho_K^{m}(\theta)d\theta \nonumber \\
  & =\frac{m}{n}\int_K |x|^{-(n-m)}dx.   \label{eqn:dual_equiv}
\end{align}

For a compact set $K \subseteq \real^n$ and $u \in S^{n-1}$, the
Steiner symmetral $S_{u} K$ with respect to the hyperplane $u^\perp$
is defined by
\begin{equation*}
S_{u} K = \big\{ 
y + t u  : y \in P_{u^{\perp}}K, \ 
 \abs{t} \leq \tfrac{1}{2} \ell_K (y)\big\},
\end{equation*}
where
\begin{equation*}
\ell_K (y) = \abs{K \cap \{y + su :  s \in \real \}}.
  \end{equation*}
We denote the rearrangement $K^*$ of $K$ to be the centered Euclidean
ball that has radius $\big(\abs{K}/\omega_n \big)^\frac{1}{n}$ so
$K^*$ has the same volume as $K$.  It is a well-known fact that there
is a sequence of directions so that successive Steiner symmetrizations
of $K$ with respect to these directions converges (in Hausdorff
distance) to $K^*$ \cite{Schneider_book}.  Steiner symmetrization of a
star body yields a star body, \cite{ZhuOrlicz}; see also \cite{Lin_Xi,
  MSY}.

Rearrangement of sets extends to integrable functions as follows: for
a non-negative function $f \in L^1 (\real^n)$, the layer cake
representation gives
\begin{equation*}
f(x) = \int_0^\infty \mathds{1}_{\{f > \alpha\}} (x) d\alpha
\qquad
x \in \real^n.
\end{equation*}
Define the symmetric decreasing rearrangment of an integrable function
$f:\mathbb{R}^n\rightarrow [0,\infty)$ by
\begin{equation*}
f^*(x) = \int_0^\infty \mathds{1}_{\{f > \alpha\}^*} (x) d\alpha
\qquad
x \in \real^n.
\end{equation*}
The function $f^*$ is radially symmetric and decreasing along each ray
emanating from the origin.  Similarly, the Steiner symmetral of an
integrable, compactly supported function $f:\mathbb{R}^n\rightarrow
[0,\infty)$ is given by
\begin{equation*}
S_uf(x)=\int_{0}^{\infty}\mathds{1}_{S_u\{f>\alpha\}}(x)d\alpha.  
\end{equation*}

Lastly, we will need some notions around marginals and the Lebesgue
differentiation theorem.  Let $f:\real^n\rightarrow [0,\infty)$ be a
bounded, integrable function with compact support and let $E$ be an
$m$-dimensional subspace with $1\leq m\leq n-1$. The marginal of $f$
on $E$ is defined for a.e. $x\in E$ by
\begin{equation}
  \label{eqn:marginal}
  \pi_E(f)(x)=\int_{E^{\perp}+x}f(y)dy.
\end{equation}
If $v_1,\ldots,v_m$ is an orthonormal basis for $E$, we let
\begin{equation*}
H_{V,t}=\{x\in \real^n:|\langle x, v_i \rangle| \leq t, i=1,\ldots,m\}
\end{equation*}
and
\begin{equation*}
  Q_{V,t}=\{x\in E:|\langle x, v_i \rangle| \leq t, i=1,\ldots,m\}.
\end{equation*}
Then as a consequence of the Lebesgue differentiation theorem, we have
\begin{equation}
  \label{eq:L_diff_thm}
  \pi_E(f)(0)= \lim_{t\rightarrow 0^{+}}
  \frac{1}{(2t)^m}\int_{Q_{V,t}}\pi_E(f)(y)dy = \lim_{t\rightarrow 0^{+}}\frac{1}{(2t)^m}\int_{H_{V,t}}f(x)dx,
\end{equation}
whenever $0$ belongs to the Lebesgue set of $\pi_E(f)$; see e.g.,
\cite{folland1999real}.

\section{Dual volumes and Gaussian slabs}

In this section we prove and expand upon the Gaussian representation
of dual volumes in Proposition \ref{prop:Tsirelson}. We start by
recalling some basic facts about Gaussian random matrices. Let $m \in
\mathbb{N}$ with $m<n$ and let $g_1,\ldots, g_m$ be independent standard
Gaussian vectors in $\real^n$ with distribution $N(0,I_n)$.  The key
property underlying Tsirelson's Gaussian representation of instrinsic
volumes \eqref{eq:Tsirel_original} is that the $m\times n$ Gaussian
matrix $G$ with rows $g_1,\ldots, g_m$ can be decomposed as
\begin{equation}
  G=(GG^T)^{1/2}Q,
\end{equation}
where $QQ^T$ is the $m\times m$ identity $I_m$ and
$Q^TQ=P_{\mathop{\rm Im}(G^T)}$ is the orthogonal projection onto the
image $\mathop{\rm Im}(G^T)$.  Moreover, $\mathop{\rm Im}(G^T)$ is
uniformly distributed on $\mathcal{G}(n,m)$ and is independent of
$(GG^T)^{1/2}$; this follows from Gram-Schmidt orthogonalization of
$g_1,\ldots,g_m$, see \cite{Tsirelson1986GeomMLE2} or \cite[\S
  4]{Paouris2013SmallBall}. As a consequence, $\mathop{\rm
  ker}(G)=\mathop{\rm Im}(G^T)^{\perp}$ is uniformly distributed on
$\mathcal{G}(n,n-m)$.  When $m=1$, i.e., $G$ is a $1\times n$ matrix,
the normalized vector $g_1/\abs{g_1}$ is uniformly distributed on the
unit sphere and is independent of $\abs{g_1}$, which amounts to polar
coordinates in $\mathbb{R}^n$ (e.g., \cite{folland1999real}).

\begin{proof}[Proof of Proposition \ref{prop:Tsirelson}]
  By the absolute continuity of the Gaussian measure, $G$ has rank $m$
  almost surely. It follows that $G^T$ can be identified with a linear
  isomorphism between $\real^m$ and the subspace $\mathop{\rm
    Im}(G^T)$ of $\real^n$.  Thus if $K$ is a star body in
  $\mathbb{R}^n$ and $G^{-T}$ is the inverse image of $G^T$, we have
  \begin{equation*}
    G^{-T}[K]=G^{-T}[K\cap \mathop{\rm Im}(G^T)].
  \end{equation*}
  It follows that
  \begin{equation*}
    \abs{G^{-T}[K]}=\mathop{\rm det}(GG^T)^{-1/2}\abs{K\cap\mathop{\rm Im}(G^T)}.
  \end{equation*}
  Since $\mathop{\rm Im}(G^T)$ is independent of $\mathop{\rm
    det}(GG^T)^{1/2}$, we have
  \begin{equation*}
    \eval_G \abs{G^{-T}[K]}=\eval_G \mathop{\rm
      det}(GG^T)^{-1/2}\eval_G \abs{K\cap\mathop{\rm Im}(G^T)}.
  \end{equation*}
  As $\mathop{\rm Im}(G^T)$ is uniformly distributed on $\mathcal{G}(n,m)$, we
  have
  \begin{equation*}
    \eval_G \abs{K\cap\mathop{\rm Im}(G^T)} = \int_{\mathcal{G}(n,m)}|K\cap E|dE, 
  \end{equation*}which implies \eqref{eq:dqmi-def1}.  Similarly, as $\mathop{\rm
    ker}(G)$ is uniformly distributed on $\mathcal{G}_{n,n-m}$ we get
  \eqref{eq:dqmi-def2}.
\end{proof}

The following proposition can be seen as an extension of Proposition
\ref{prop:Tsirelson} involving the Gaussian slabs $H_{G,t}$
(cf. \eqref{eq:HGt}) rather than slices, and justifies the limiting
relation asserted in \eqref{eqn:lim}.

\begin{prop}
  \label{lem:vm-gaussian}
  Let $m,n \in \mathbb{N}$, $m<n$ and let $K$ be a star body in
  $\mathbb{R}$. Then for $\ell=n-m$,
  \begin{equation}
    \label{eq:vm-rep}
    \tildee{V}_{\ell} (K)
    =\lim_{t \to 0^{+}} \lim_{N\rightarrow \infty}\frac{1}{(2t)^m\Delta_m} \tildee{V}_{\ell,t,N}(K)
  \end{equation}  
\end{prop}


We will need the following basic lemma giving the order of
$\Delta_m$; a proof is given for completeness.

\begin{lemma}Let $m,n\in\mathbb{N}$, $m<n$. Then
  $\Delta_m$ defined in \eqref{eqn:Delta} is given by
  \begin{equation*}
    \Delta_m=\frac{1}{(2\pi)^{m/2}}\cdot\frac{n}{n-m}\cdot \frac{\omega_n}{\omega_{n-m}}.
  \end{equation*}
\end{lemma}

\begin{proof}It is known that $\mathop{\rm det}(GG^*)$ has the same distribution
  as the product of independent $\chi^2$ random variables
  $\chi_n^2\chi_{n-1}^2\cdots \chi_{n-m+1}^2$, with $\chi_k^2$ having
  $k$ degrees of freedom, for $k=n,\ldots,n-m+1$ (e.g.,
  \cite[Theorem 7.5.1]{Anderson}).  For each such $k$, we use polar
  coordinates
  \begin{align*}
    \eval (\chi_k^2)^{-1/2} &= \frac{k\omega_k}{(2\pi)^{k/2}}\int_0^{\infty}r^{k-2}e^{-r^2/2}dr\\
    & = \frac{k\omega_k}{2^{3/2}\sqrt{\pi}}\cdot\frac{\Gamma{\left(\tfrac{k-1}{2}\right)}}{\pi^{(k-1)/2}}\\
    & =\frac{1}{(2\pi)^{1/2}}\cdot\frac{k}{k-1}\cdot\frac{\omega_k}{\omega_{k-1}},
  \end{align*}where we used $\omega_{k-1}=\pi^{(k-1)/2}/\Gamma\left(\tfrac{k-1}{2}+1\right)$ in the last step.
  Using independence
  \begin{equation*}
    \eval_G\mathop{\rm det}(GG^*)^{-1/2}=\prod_{k=n-m+1}^n \eval (\chi_k^2)^{-1/2}=\prod_{k=n-m+1}^n
    \frac{1}{(2\pi)^{1/2}}\cdot\frac{k}{k-1} \cdot \frac{\omega_k}{\omega_{k-1}},
  \end{equation*}which, upon cancellation, gives the result.
  
\end{proof}

It will be convenient to use the following notational
abbreviation for the indicator functions appearing in
$\tildee{V}_{\ell,t,N}(K)$ (cf. \eqref{eq:edqmi-def}):
\begin{equation*}
  \mathds{1}_{H_{G,t}}(X_i)=[X_i\in H_{G,t}].
\end{equation*}

\begin{proof}[Proof of Proposition \ref{lem:vm-gaussian}]
  Let $X_1,X_2,\ldots$ be independent random vectors with density
  $\frac{1}{\abs{K}}\mathds{1}_{K}$.  Let $G$ be an $m \times n$
  Gaussian matrix with rows $g_1,\ldots, g_m$. When $G$ is fixed, the
  law of large numbers (e.g., \cite[pg. 391]{Shir_book}) for $\{X_i\}$ gives
  \begin{equation*}
    \lim_{N\rightarrow \infty} \left(\frac{\abs{K}}{N}\sum_{i=1}^N[X_i\in H_{G,t}] \right)
    = \int_{\mathbb{R}^n}[x\in H_{G,t}]\mathds{1}_K(x)dx.
  \end{equation*}
  By dominated convergence, for $t>0$, 
  \begin{equation*}
    \lim_{N\rightarrow \infty}\tildee{V}_{\ell, t, N}(K) = \lim_{N\rightarrow \infty} \frac{\omega_n}{\omega_{n-m}}\eval_G
    \left(\frac{\abs{K}}{N}\sum_{i=1}^N[X_i\in H_{G,t}] \right) =
    \frac{\omega_n}{\omega_{n-m}}\eval_G \int_{\mathbb{R}^n}[x\in H_{G,t}]\mathds{1}_K(x)dx.
    \end{equation*}
  Thus it suffices to show that
  \begin{equation}
    \label{eqn:slabtoslice}
\tildee{V}_{\ell}(K)= \lim_{t\rightarrow 0^{+}}
\frac{1}{(2t)^m\Delta_m}\frac{\omega_n}{\omega_{n-m}} \eval_G
\int_{\mathbb{R}^n} [x \in H_{G,t}]\mathds{1}_K(x)dx.
  \end{equation}
  For $x\in \mathbb{R}^n\backslash\{0\}$ and $t>0$, we
  define
  \begin{equation}
    \label{eqn:slab}
    H_{x,t}=\{y\in \mathbb{R}^n: \lvert \langle x,y\rangle
    \rvert\leq t\}.
    \end{equation}
    Note that
  \begin{equation*}
    [x\in H_{G,t}] = \prod_{j=1}^m[g_j\in
      H_{x,t}]=\prod_{j=1}^m\left[g_j\in
      H_{\tfrac{x}{\abs{x}},\tfrac{t}{\abs{x}}}\right].
  \end{equation*}
  By independence of $g_1,\ldots,g_m$,  for $x\not=0$ we have
  \begin{align*}
    \eval_G [x\in H_{G,t}] & = \gamma^m_n\left(H_{\frac{x}{\lvert x \rvert }, \frac{t}{\lvert x
        \rvert}}\right)\\
    & = \left(\frac{2}{\sqrt{2\pi}}\int_0^{t/\lvert x\rvert}
    e^{-r^2/2}dr\right)^m\\
    & = \left(\frac{2}{\sqrt{2\pi}}\frac{t}{\abs{x}}\int_{0}^1
      \exp\left(-\tfrac{t^2 u^2}{2\abs{x}^2}\right)du\right)^m.
  \end{align*}
  It follows that
  \begin{align*}
    \frac{1}{(2t)^m}\eval_G
    \int_{\mathbb{R}^n}[x\in H_{G,t}] \mathds{1}_K(x)dx 
    & =  \int_{\mathbb{R}^n}
    \left(\frac{1}{\sqrt{2\pi}}\frac{1}{\abs{x}}\int_{0}^1
      \exp\left(-\tfrac{t^2 u^2}{2\abs{x}^2}\right)du\right)^m \mathds{1}_K(x)dx.
  \end{align*}
Since $x\mapsto \tfrac{1}{\abs{x}}$ is locally integrable (for
$n\geq2$), we get by dominated convergence,
\begin{align*}
  \lim_{t\rightarrow 0^{+}}  \frac{1}{(2t)^m\Delta_m}\frac{\omega_n}{\omega_{n-m}}\eval_G
    \int_{\mathbb{R}^n}[x\in H_{G,t}] \mathds{1}_K(x)dx 
     = \frac{n-m}{m} \int_{K}\abs{x}^{-m}dx
     = \tildee{V}_{\ell}(K),
  \end{align*}where we used \eqref{eqn:dual_equiv} in the last step.

\end{proof}

\begin{rmk}
  The limiting relation \eqref{eqn:slabtoslice} can also be proved by
  applying the Gram--Schmidt algorithm to $g_1,\ldots,g_m$ and using
  the Lebesgue differentiation theorem \eqref{eq:L_diff_thm}. The
  proof we have given above aligns naturally with the Gaussian
  representation for $\tildee{V}_{\ell,t,N}(K,q)$ that we develop in
  the next section.
\end{rmk}

\section{Empirical functionals and symmetrization: main proofs}

In this section, we complete the proofs of Theorems \ref{thm:eqmi} and
\ref{thm:moments}. In order to make a clear distinction between the
tools for each theorem, we prove them separately (even though the
method for the former is subsumed in the latter).  As above, we use
$H_{x,t}$ for the slab defined in \eqref{eqn:slab}.

\begin{lemma}
  \label{lem:slab-dom}
  Let $X$ and $X^*$ be random vectors with densities $f$ and $f^*$,
  respectively.  Then for each $t>0$,
  \begin{equation*}
    \gamma_n (H_{X,t})\prec \gamma_n (H_{X^*, t}).
  \end{equation*}
\end{lemma}

\begin{proof} Let $s > 0$. For $z\in \mathbb{R}$, let
  $$\Phi(z)=\frac{1}{(2\pi)^{1/2}}\int_{-\infty}^{z}e^{-r^2/2}dr.$$
  By rotational invariance of the Gaussian measure,
  \begin{align*}
    \prob \big( \gamma_n (H_{X,t}) > s \big) & = \prob \left(
    \gamma_1 \left( \left[ - \tfrac{t}{\abs{X}}, \tfrac{t}{\abs{X}} \right]
    \right) > s \right)\\
    & = \prob\left(2 \Phi \left(\tfrac{t}{\abs{X}} \right) - 1>s\right)\\
    & = \prob\left(\abs{X}<c(s,t)\right),
  \end{align*}where $c(s,t)=t(\Phi^{-1}\big((1+s)/2\big))^{-1}.$
  We can conclude by applying the simplest of rearrangement
  inequalities (e.g., in \cite[Theorem 3.4]{LL}),
  \begin{align*}
    \prob (\abs{X} < c(s,t)) &= \int_{\real^n} \mathds{1}_{c(s,t)B_2^n} (x)
    f(x) \, dx 
    \leq \int_{\real^n} \mathds{1}_{(c(s,t) B_2^n)^*} (x) f^* (x) \, dx 
    = \prob (\abs{X^*}\leq c(s,t)).\qedhere
  \end{align*}
\end{proof}

We will use the following tensorization property of the stochastic
ordering \eqref{eq:stoch_dom}, which can be proved by a standard
conditioning argument; see \cite{SS12} for general theory on
stochastic orderings.

\begin{lemma}
  \label{lem:tensor}
  Let $\{\xi_i\}_{i=1}^N$ and $\{\eta_i\}_{i=1}^N$ be collections of
  independent non-negative random variables. Assume that
  $\xi_i\prec\eta_i$ for each $i=1,\ldots,N$. Then for all
  $t_1,\ldots,t_N\in(0,\infty)$,
  \begin{equation*}
    \sum_{i=1}^n  t_i\xi _i \prec
     \sum_{i=1}^n t_i\eta_i.
  \end{equation*}
\end{lemma}

\begin{proof}[Proof of Theorem \ref{thm:eqmi}]
  Fix $m <n$ and let $K$ be a star body in $\real^n$.  Let
  $X_1,\ldots, X_N$ be independent random vectors uniformly
  distributed in $K$. Let $g_1,\ldots, g_m$ denote the rows of
  $G$. Then for each $t$ and $N$,
  \begin{align*}
    \tildee{V}_{\ell, t, N}(K) &= \frac{\abs{K}}{N}\eval_G \sum_{i=1}^N [X_i\in H_{G,t}] \\
    & = \frac{\abs{K}}{N}\eval_G \sum_{i=1}^N  \prod_{j=1}^m [g_j\in H_{X_i,t}] \\
    & = \frac{\abs{K}}{N}\sum_{i=1}^N \prod_{j=1}^m \gamma_n (H_{X_i,t})  \\ &=
    \sum_{i=1}^N \gamma_n^m (H_{X_i,t}).
  \end{align*}
  The same identity applies to indepdent random vectors $X_i^*$
  sampled uniformly in $K^*$.  Thus by Lemmas \ref{lem:slab-dom} and
  \ref{lem:tensor}, we have
  \begin{equation*}
    \tildee{V}_{m,t,N}(K)=\frac{\abs{K}}{N}\eval_G \sum_{i=1}^N
           [X_i\in H_{G,t}] \prec \frac{\abs{K}}{N}\eval_G
           \sum_{i=1}^N [X_i^*\in H_{G,t}]= \tildee{V}_{m,t,N}(K^*).\qedhere
  \end{equation*}
\end{proof}

The proof of Theorem \ref{thm:moments} relies on a stronger
symmetrization result from \cite{CEFPP15}, which makes use of Christ's
version \cite{Christ1984} of the Rogers/Brascamp--Lieb--Luttinger
inequality \cite{Rogers1957, BLL1974}; see the survey \cite{PP17} for
background and its use in stochastic geometry. We formulate it here in
the same notation as Lemma \ref{lem:slab-dom} (see \cite[Theorem 1.2]{CEFPP15}).

\begin{thm}
  \label{thm:CEFPP}
  Let $\{X_i\}_{i=1}^{N}$ and $\{X_i^{*}\}_{i=1}^N$ be collections of
  independent random vectors with densities $f$ and $f^*$,
  respectively.  Then for any $t>0$,
  \begin{equation*}
    \gamma_n\left((\mathop{\rm conv}\{\pm t^{-1}X_{i}\})^{\circ}\right)
    \prec \gamma_n\left((\mathop{\rm conv}\{\pm t^{-1}X_i^*\})^{\circ}\right).
  \end{equation*}
\end{thm}

\begin{rmk}
  For a single random vector $X$, we have
  \begin{equation*}
    \mathop{\rm conv}\{\pm t^{-1}X_1\}^{\circ} = [-t^{-1}X,t^{-1}X]^{\circ}=H_{X,t},
  \end{equation*}hence  Theorem \ref{thm:CEFPP} covers Lemma \ref{lem:slab-dom}.
\end{rmk}

\begin{proof}[Proof of Theorem \ref{thm:moments}]
  Fix $t>0$ and $q\in \mathbb{N}$. We use multinomial expansion to write
\begin{align*}
  \tildee{V}_{\ell,t,N}(K,q)^q & = \eval_G
  \left[\frac{1}{N}\sum_{i=1}^N[X_i\in
      H_{G,t}]\right]^q \\ & = \frac{1}{N^q}\eval_G \left[ 
     \sum_{I} \prod_{k=1}^q [X_{i_k}\in
      H_{G,t}] \right],
\end{align*}
where the last sum ranges over all $I=(i_1, \ldots,
i_q)\in\{1,\ldots,N\}^q$.  For such fixed $I$, we write
$X(I)=(X_{i_1},\ldots,X_{i_q})$ and define
\begin{align*}
  H_{X(I),t}&=\{y\in \mathbb{R}^n:|\langle X_{i_k},y \rangle|\leq
  t, \; k=1,\ldots,q\}
\end{align*}
so that
\begin{equation*}
  \prod_{k=1}^q[X_{i_k}\in H_{G,t} ]=\prod_{j=1}^m[g_j\in H_{X(I),t}].
\end{equation*}
Next, we note that
\begin{equation*}
  H_{X(I),t}=(\mathop{\rm conv}\{\pm t^{-1}X_{i_1},\ldots,\pm t^{-1}X_{i_q}\})^{\circ}.
\end{equation*}
Theorem \ref{thm:CEFPP} implies that for each $I$, 
\begin{equation*}
\gamma_n^m \big( ( \mathop{\rm conv}\{\pm
t^{-1}X_{i_k}\}_{k=1}^q)^\circ \big)
\prec \gamma_n^m \big( ( \mathop{\rm conv}\{\pm
  t^{-1}X_{i_k}^{\ast}\}_{k=1}^q)^\circ \big),
\end{equation*}which carries over to expectations.
Denoting the expectations in $\{X_i\}$ and $\{X_i^*\}$ as $\eval_{\bf
  X}$ and $\eval_{\bf X^*}$, respectively, we conclude with
\begin{align*}
\eval_{\bf X} \tildee{V}_{\ell,t,N}(K,q)^q & = \frac{\abs{K}^q}{N^q} \sum_I \eval_{{\bf
    X}}\gamma_n^m \big( \mathop{\rm conv}\{\pm
t^{-1}X_{i_k}\}_{k=1}^q)^\circ \big)\\ & \leq \frac{\abs{K}^q}{N^q} \sum_I
\eval_{{\bf X^*}}\gamma_n^m \big( ( \mathop{\rm conv}\{\pm
t^{-1}X_{i_k}^*\}_{k=1}^q)^\circ \big)\\ & = \eval_{\bf X^*}
\tildee{V}_{\ell,t,N}(K^*,q)^q.\qedhere
\end{align*}

\end{proof}

\section{The Busemann intersection inequality via Gaussian slabs}

In this section, we provide the proof of Theorem \ref{thm:ebie}. The
main step is the following proposition.

\begin{prop}
  \label{prop:mom-limit}
  Let $K$ be a star body in $\real^n$. Let $g$ be a standard Gaussian
  random vector in $\real^n$ and let
  $c_n=n\omega_n/(2\pi)^{n/2}$. Then
  \begin{equation}
    \label{eq:slabtoslice}
    \lim_{t \to 0^{+}}
    \eval_g \left[ \frac{1}{c_n\ln(1/t)}
      \left( \frac{|K \cap H_{g,t}|}{2t} \right)^n \right]
    = \eval_g |K \cap g^\perp|^n.
  \end{equation}
\end{prop}

We first note that \eqref{eq:L_diff_thm} implies that for fixed $g$,
\begin{equation*}
  \frac{\lvert K\cap H_{g,t}\rvert}{t}
  \rightarrow \lvert g\rvert^{-1}\lvert K\cap g^{\perp}\rvert;
\end{equation*}
moreover, $|g|$ and $g^{\perp}$ are independent. However, Proposition
\ref{prop:mom-limit} involves $n$-th powers and $\lvert g\rvert^{-n}$
is not integrable. We will show that a careful regularization of the
divergent factor $\eval_g \lvert g \rvert^{-n}$ leads to the $\ln (1/t)$
term in the limit \eqref{eq:slabtoslice}.

We will first prove several auxilliary lemmas. For each $t>0$, we
define
\begin{equation*}
  \gamma_{t} =\eval_g \lvert g \rvert^{-n}\mathds{1}_{\{\abs{g}>t\}}.
\end{equation*}
We start with the limiting behavior of $\gamma_t$.
\begin{lemma}
  \label{lem:lim}
  Let $n\in \mathbb{N}$ and set $c_n=n\omega_n/(2\pi)^{n/2}$. Then
  \begin{equation*}
    \lim_{t\rightarrow 0^{+}}\frac{\gamma_t}{c_n\ln(1/t)}=1,
  \end{equation*}and for any $\delta\in (0,1]$,
    \begin{equation*}
      \lim_{t\rightarrow 0^{+}} \frac{\gamma_{t}}{\gamma_{t/\delta}} =1.
    \end{equation*}
    
\end{lemma}  

\begin{proof}Using polar coordinates, for any $\delta\in (0,1]$, 
   \begin{equation*}
     \gamma_{t/\delta}=\eval_g |g|^{-n} \mathds{1}_{\{\abs{g}>t/\delta\}} = \frac{n
       \omega_n}{(2\pi)^{n/2}} \int_{t/\delta}^\infty r^{-1} e^{-r^2/2} \, dr
\end{equation*}and 
\begin{equation*}
\int_{t/\delta}^\infty r^{-1} e^{-r^2/2} \, dr
= (-\ln(t/\delta))e^{-t^2/(2\delta^2)}
	+\int_{t/\delta}^\infty (r\ln r)(e^{-r^2/2}) \, dr.
\end{equation*}Thus for $\delta =1$, we have
\begin{equation*}
  \frac{\gamma_t}{c_n \ln(1/t)}=e^{-t^2/2}+\frac{1}{
    \ln(1/t)}\int_t^\infty (r\ln r)(e^{-r^2/2}) \, dr,
\end{equation*}which tends to $1$ as $t\rightarrow 0^{+}$
since $r\mapsto r \ln r$ is integrable on $(0,1)$.

To prove the second equality, write
\begin{align*}
  \frac{\gamma_{t}}{\gamma_{t/\delta}} =
  1 + \frac{c_n}{\gamma_{t/\delta}} \int_{t}^{t/\delta}
  r^{-1} e^{-r^2/2} \, dr
\end{align*}
and observe that
\begin{align*}
  \frac{1}{\gamma_{t/\delta}} \int_{t}^{t/\delta} r^{-1} e^{-r^2/2} \, dr
  \leq \frac{1}{\gamma_{t/\delta}}\int_{t}^{t/\delta}
  r^{-1} \, dr
  &\leq \frac{\ln(1/\delta)}{\gamma_{t/\delta}},
\end{align*}
which tends to $0$ as $t\rightarrow 0^{+}$.
\end{proof}

We will condition on $\abs{g}$ according to
\begin{equation}
  \eval_g \abs{ K\cap H_{g,t}}^n =
  \eval_g \abs{ K\cap H_{g,t}}^n\mathds{1}_{\{\abs{g} \leq t\}} +
    \eval_g \abs{K\cap H_{g,t}}^n \mathds{1}_{\{\abs{g} > t\}},
\end{equation}and treat each summand separately.

\begin{lemma}
  \label{lem:g<t}
  For any star body $K$ in $\mathbb{R}^n$,
  \begin{equation}
    \lim_{t\rightarrow 0^{+}}\frac{\eval_g \abs{K \cap H_{g,t}}^n
      \mathds{1}_{\{\abs{g}\leq t\}} }{ (2t)^n \, \gamma_{t}}=0.
  \end{equation}
\end{lemma}

\begin{proof}Using polar coordinates with $c_n=n\omega_n/(2\pi)^{n/2}$,
  \begin{align*}
\eval_g \abs{K \cap H_{g,t}}^n
  \mathds{1}_{\{\abs{g}\leq t\}} \leq \abs{K}^n \prob \big( \abs{g} \leq
    t\big)
    = c_n \abs{K}^n  \int_0^{t}
    r^{n-1} e^{-r^2/2}\, dr
     \leq  \frac{c_n \abs{K}^n}{n} \cdot t^n.
  \end{align*}
Thus
  \begin{align*}
     \frac{\eval_g  \abs{K \cap H_{g,t}}^n
      \chi_{\{\abs{g}\leq t\}} }{ (2t)^n \, \gamma_{t}}
 \leq \frac{c_n\abs{K}^n}{n} \cdot \frac{t^n}{(2t)^n\gamma_{t}}
  \leq \frac{c_n\abs{K}^n}{2^nn} \cdot
    \frac{1}{\gamma_{t}},
  \end{align*}which tends to $0$ as $t\rightarrow 0^{+}$.
  \end{proof}

The slab volume $\abs{K\cap H_{g,t}}$ can be computed by integrating
the marginal of $\mathds{1}_K$ on the subspace $\mathop{\rm
  span}\{g\}$ (cf. \eqref{eqn:marginal}). For brevity, we write
$\theta=g/|g|$ and set $[\theta]=\mathop{\rm span}\{\theta\}$. To
treat the $n$-th power, we introduce the function
$F_{K,\theta}:\mathbb{R}^n\rightarrow [0,\infty)$ defined as
  \begin{equation}
    \label{eqn:FK}
    F_{K,\theta} (s) = \pi_{[\theta]}(\mathds{1}_K)(s_1 \theta)\cdots
    \pi_{[\theta]}(\mathds{1}_K)(s_n \theta).
  \end{equation}
  With this notation, 
  \begin{equation}
    \label{eqn:FK0}
  F_{K,\theta}(0)=\abs{K\cap\theta^{\perp}}^n.
  \end{equation}

  \begin{lemma}
    \label{lem:g>t}
    Let $K$ be a star body in $\mathbb{R}^n$.  Let $g$ be a Gaussian
    vector in $\mathbb{R}^n$ and set $\theta = g/\lvert g \rvert$. Let
    $F_{K,\theta}$ be as defined in \eqref{eqn:FK}.  Define $\psi:\mathbb{R}^n\rightarrow [0,\infty)$
      by
    \begin{align}
      \label{eqn:psi_t}
    \psi_t (s) := \eval_{\abs{g}} \mathds{1}_{\left[-\frac{t}{|g|}, \frac{t}{|g|}\right]^n} (s)
    \, 	\mathds{1}_{\{\abs{g} > t\}}.
  \end{align}
  Then
  \begin{equation*}
    \eval_g \abs{K \cap H_{g,t}}^n \, \mathds{1}_{\{\abs{g} > t\}} 
    =\eval_\theta \big( F_{K,\theta}*\psi_t \big) (0).
  \end{equation*}

\end{lemma}

\begin{proof}Assume first that $g$ is fixed. Then
\begin{align*}
  \abs{K \cap H_{g,t}}^n &=\prod_{i=1}^n\int_{\real^n}
  \mathds{1}_K (x_i) \mathds{1}_{H_{g,t}} (x_i) \, dx_{i}.
\end{align*}
For each $i$, we decompose $x_i\in \mathbb{R}^n$ as $x_i=y_i+z_i$ with
$y_i\in [\theta]$, $z_i\in \theta^{\perp}$ and use $\langle y_i+z_i,
g\rangle = \langle y_i, g \rangle$ along with Fubini's theorem to get
\begin{align*}
  \int_{\mathbb{R}^n} \mathds{1}_K(x_i)\mathds{1}_{H_{g,t}}(x_i) dx_i & =
  \int_{[\theta]} \int_{\theta^{\perp}} \mathds{1}_{K}(y_i+z_i)\mathds{1}_{H_{g,t}}(y_i+z_i) dy_i dz_i\\
  & = \int_{[\theta]} \pi_{[\theta]}(\mathds{1}_K)(y_i) \mathds{1}_{H_{g,t}}(y_i)dy_i\\
  & = \int_{\mathbb{R}} \pi_{[\theta]}(\mathds{1}_K)(s_i\theta)
  \mathds{1}_{\left[\frac{-t}{\lvert g \rvert},\frac{t}{\lvert g\rvert}\right]}(s_i)ds_i,
\end{align*}where $y_i=s_i \theta$ and we have used $\langle y_i, g \rangle =
s_i\lvert g\rvert$.  Using Fubini's theorem again, and writing
$s=(s_1,\ldots,s_n)$ and $ds=ds_1\cdots ds_n$, we have
\begin{align*}
  \abs{K \cap H_{g,t}}^n & =
   \int_{\mathbb{R}^n} F_{K,\theta}(s)
   \mathds{1}_{\left[-\frac{t}{\lvert g \rvert},\frac{t}{\lvert g\rvert} \right]^n}(s)ds.
\end{align*}
Note that $F_{K,\theta}$ depends only on $\theta$ and not on $|g|$.
As $\theta$ and $\lvert g \rvert$ are independent,
we can decompose $\eval_{g}$ as $\eval_{\theta}\eval_{\abs{g}}$
and use Fubini's theorem once more to conclude
\begin{align*}
  \eval_\theta \eval_{\abs{g}} \abs{K \cap H_{g,t}}^n \mathds{1}_{\{\abs{g} > t\}}
  & = \eval_{\theta}\int_{\mathbb{R}^n} F_{K,\theta}(s) \eval_{\lvert
    g\rvert} \mathds{1}_{\left[-\frac{t}{\lvert g
        \rvert},\frac{t}{\lvert g\rvert} \right]^n}(s) \mathds{1}_{\{\abs{g} \geq t\}}ds\\
  & = \eval_\theta \big( F_{K,\theta} * \psi_t
  \big) (0).\qedhere
\end{align*}

\end{proof}

We will use the following fundamental theorem on convergence of
approximate identities (see \cite[pg. 27]{Graf14}).
  
\begin{thm}
  \label{thm:approx-id-conv}
  Let $\{\varphi_t\}_{t\in (0,1)}$ be an approximate identity $\real^n$,
  i.e. a family of functions $\varphi_t:\mathbb{R}^n\rightarrow
  \mathbb{R}$ with $\sup_{t\in (0,1)} \int \lvert \varphi _t\rvert
  <\infty$ such that
  \begin{enumerate}\item[(a)] for all $t\in (0,1)$, 
    $\int_{\real^n} \varphi_t  = 1$;
  \item[(b)] for any $\delta > 0$,
    \begin{equation}
      \label{eqn:delta_nbhd}
      \lim_{t \to 0^{+}} \int_{\real^n \setminus \delta B_\infty^n}
      \varphi_t (s) \, ds = 0.
    \end{equation}
  \end{enumerate}
    Then for any $f\in L^{\infty}(\mathbb{R}^n)$ such that $f$ is
    continuous at $0$,
    \begin{equation*}
      (\varphi_t*f)(0)\rightarrow f(0).
    \end{equation*}
 \end{thm}
    
We are now ready to prove the proposition.

\begin{proof}[Proof of Proposition \ref{prop:mom-limit}]
  By Lemma \ref{lem:g<t}, it suffices to show
    \begin{equation}
      \label{eqn:goal}
    \lim_{t\rightarrow 0^{+}}\frac{\eval_g \abs{K \cap H_{g,t}}^n
      \mathds{1}_{\{\abs{g}>t\}} }{ (2t)^n \, \gamma_{t}}= \eval_g
    \abs{K\cap g^{\perp}}^n.
  \end{equation}
    With $\psi_t$ given by \eqref{eqn:psi_t}, we let
    \begin{equation*}
      \beta_t = \int_{\mathbb{R}^n} \psi_t(s)ds =
      \int_{\real^n} \eval_{|g|}  \mathds{1}_{\left[-\frac{t}{|g|}, 
	\frac{t}{|g|}\right]^n} (s)  \mathds{1}_{\{\abs{g} > t\}} \,ds
      \end{equation*}
and note that
\begin{align*}
  \beta_t & = \eval_{|g|} \big| \left[ -\tfrac{t}{\abs{g}},\tfrac{t}{\abs{g}}\right]^n\big| \,
  \mathds{1}_{\{\abs{g} > t\}}  = (2t)^n \, \eval_{|g|} |g|^{-n} \,
  \mathds{1}_{\{\abs{g} > t\}} = (2t)^n \, \gamma_{t}.
\end{align*}
Set $\tildee{\psi}_t = \frac{1}{\beta_t}\psi_t$.  Using Lemma
\ref{lem:g>t} and \eqref{eqn:FK0}, our goal \eqref{eqn:goal} can be
restated as
\begin{equation}
  \label{eqn:restate_goal}
  \lim_{t \rightarrow 0^{+}}\eval_{\theta}
  (F_{K,\theta}*\tildee{\psi}_t)(0) = \eval_{\theta} F_{K,\theta}(0).
\end{equation}

We will show that $\tildee{\psi}_t$ is an approximate identity and use
Theorem \ref{thm:approx-id-conv}.  By construction $\int
\tildee{\psi_t} =1$. To see that condition \eqref{eqn:delta_nbhd} is
satisfied, we fix $\delta>0$. We have
  \begin{align*}
    \int_{\real^n \setminus \delta B_\infty^n} \psi_t (s) \, ds
    &=  \eval_{|g|} \int_{\real^n \setminus 
      \delta B_\infty^n} \mathbb{1}_{\left[-\frac{t}{|g|}, 
	\frac{t}{|g|}\right]^n} (s) \, ds
    \, \mathds{1}_{\{\abs{g} > t\}},
  \end{align*}and for $s\in \mathbb{R}^n \backslash \delta B_{\infty}^n$, 
  \begin{align*}
    \mathds{1}_{\left[-\frac{t}{\lvert g \rvert},\frac{t}{\lvert g
          \rvert}\right]^n}(s) = \begin{cases} 0 & \text{ if }
      \frac{t}{\lvert g \rvert} \leq \delta\\ 1 & \text{ if }
      \frac{t}{\lvert g \rvert} > \delta.
    \end{cases}
  \end{align*}
Thus
  \begin{align*}
    \int_{\real^n \setminus \delta B_\infty^n} \psi_t (s) \, ds
    & = \eval_{|g|} \int_{\real^n} \mathbb{1}_{\left[-\frac{t}{|g|}, 
	\frac{t}{|g|}\right]^n \; \setminus \;
      [-\delta,\delta]^n} (s) \, ds
    \, \mathds{1}_{\{t<\abs{g}<t/\delta\}}\\
    &=  \eval_{|g|} \left[ \left( \left( 
      \tfrac{2t}{\abs{g}}\right)^n - (2\delta)^n \right) \,
      \mathds{1}_{\{t < \abs{g} < t/{\delta}\}} \right]\\
    & =  (2t)^n\eval_{\abs{g}} \big[ |g|^{-n}
      \mathds{1}_{\{t < \abs{g} < t/\delta\}} \big]
    - (2\delta)^n \prob \left( t < \abs{g} < t/\delta \right).
  \end{align*}
  Using $\beta_t=(2t)^n\gamma_t$, the normalized sequence
  $\tildee{\psi_t}$ satisfies
  \begin{align*}
    \int_{\real^n \setminus \delta B_\infty^n} \tildee{\psi_t} (s) \, ds
    &= \frac{\gamma_{t} - \gamma_{t/\delta}}{\gamma_t}
    - \frac{(2\delta)^n}{(2t)^n\gamma_t} \prob \left( t < |g| < t/\delta \right).
  \end{align*}
  By Lemma \ref{lem:lim}, the first term tends to $0$ as $t\rightarrow
  0^{+}$. The second term can be estimated as
    \begin{align*}
      \frac{\delta^n}{t^n\gamma_t}\prob \left( t < |g| < t/\delta \right)
      & = \frac{c_n\delta^n}{t^n\gamma_t}
      \int_{t}^{t/\delta} r^{n-1} e^{-r^2/2} \, dr\\
     &\leq
    \frac{c_n \delta^n}{t^n \gamma_t}\cdot
     \frac{1}{n} \left(\frac{t^n}{\delta^n} - t^n\right)\\
     & =  \frac{c_n\delta^n}{n\gamma_t}
     \left(\frac{1}{\delta^n} - 1\right),
  \end{align*}which vanishes as $t\rightarrow 0^{+}$ since $\gamma_t$ is of order $\ln(1/t)$.
    Therefore $\{\tildee{\psi}_t\}_{t>0}$ is an approximate identity.
    on $\mathbb{R}^n$.

Finally, as $K$ is compact, the marginals
$\pi_{[\theta]}(\mathds{1}_K)$ are bounded independently of $\theta$,
hence the same is true of $F_{K,\theta}$. Since
$\{F_{K,\theta}\}_{\theta \in S^{n-1}}$ is uniformly bounded and $\int
\tildee{\psi}_t=1$ for each $t$, the family
$\{F_{K,\theta}*\tildee{\psi_t}:\theta\in S^{n-1}, t\in(0,1)\}$ is
uniformly bounded. Thus by dominated convergence and Theorem
\ref{thm:approx-id-conv},
\begin{align*}
  \lim_{t \to 0} \eval_\theta  \big(F_{K,\theta} * \tildee{\psi}_t \big) (0) = 
  \eval_\theta \lim_{t \to 0} \big(F_{K,\theta} * \tildee{\psi}_t \big) (0) 
  = \eval_\theta F_{K,\theta} (0).
\end{align*}which establishes \eqref{eqn:restate_goal} and completes the proof.
\end{proof}

\begin{proof}[Proof of Theorem \ref{thm:ebie}]
    The inequalities \eqref{eqn:bie-steiner} and \eqref{eq:bie-emp}
    are immediate from Theorem \ref{thm:moments}. We need only verify
    \eqref{eqn:ebie_lim}. By an application of the law of large
    numbers and dominated convergence, for each $t>0$, when
    $N\rightarrow \infty$, we have
    \begin{equation*}
      \lim_{N\rightarrow \infty} \tildee{\Phi}_{t,N} (K) = \eval_g |K\cap
      H_{g,t}|^n,
    \end{equation*}
    An application of Proposition \ref{prop:mom-limit} concludes the
    proof.
    \end{proof}

\subsubsection*{Concluding remarks}

We have defined $\tildee{V}_{\ell,t,N}(K)$ for star bodies $K$ as that
is the most natural setting within dual Brunn--Minkowski theory.  The
proof of Theorem \ref{thm:CEFPP} shows that for $t>0$, $N\in
\mathbb{N}$ and $q\in \mathbb{N}$,
\begin{equation*}
  \eval_{\bf X}\eval_G
  \left(\frac{\abs{K}}{N}\sum_{i=1}^N[X_i\in H_{G,t}]\right)^q \leq
  \eval_{\bf X^*}\eval_G
  \left(\frac{\abs{K}}{N}\sum_{i=1}^N[X_i^*\in H_{G,t}]\right)^q,
  \end{equation*}
 whenever $X_i$ and $X_i^*$ have densities $f$ and $f^*$,
respectively.  Deriving results about sections from the Gaussian slabs
$H_{G,t}$ ultimately relies on some regularity, e.g., when
$f_K=\frac{1}{\abs{K}}\mathds{1}_K$ we used continuity of the marginal
distributions $\pi_E(f_K)$ at the origin in the proof of Proposition
\ref{prop:mom-limit}. However, Theorem \ref{thm:approx-id-conv} is
just one example of a result on approximate identities that could
naturally interface with the empirical approach beyond star bodies.

\subsection{Acknowledgements.}
The first-named author gratefully acknowledges support from the NSF
under grant DMS-2405441 and from AFOSR under grant FA9550-25-1-0284.
The second-named gratefully acknowledges support from NSF Grant
DMS-2105468.
\bibliographystyle{plain} \bibliography{qmi}

\vspace{1cm}

\noindent Grigoris Paouris, Department of Mathematics, Texas A\&M University, and
Department of Mathematics, Princeton University, {\tt grigoris@tamu.edu}\\

\noindent Peter Pivovarov, Mathematics Department, University of Missouri-Columbia, {\tt pivovarovp@missouri.edu}\\

\noindent Paul Simanjuntak, Department of Mathematics, Texas A\&M University,
{\tt simanjuntak@tamu.edu}

\end{document}